\documentclass[a4paper, 11pt]{article}	 

\usepackage[paper=a4paper,left=1.8cm,right=1.8cm]{geometry}
\usepackage[utf8]{inputenc}
\usepackage[T1]{fontenc}
\usepackage{amsmath}
\usepackage{amssymb}
\usepackage{amsthm}
\usepackage{color}
\usepackage[onehalfspacing]{setspace}

\allowdisplaybreaks

\usepackage{authblk}

\renewcommand{\i}{\mathrm{i}}
\newcommand{\xCrit}{\tilde x^{(c)}}
\newcommand{\alphCrit}{\alpha^{(c)}}

\newcommand{\x}{x}
\newcommand{\f}{f}

\renewcommand{\i}{\mathrm{i}}

\renewcommand{\r}{r}

\newcommand{\B}{B}
\newcommand{\I}{I}
\newcommand{\balpha}{\alpha}
\newcommand{\R}{\mathbb{R}}

\newcommand{\ba}{a}
\newcommand{\bb}{b}
\newcommand{\xE}{\tilde x}

\newcommand{\bkappa}{\kappa}
\newcommand{\expsin}{\mathrm{s}}
\newcommand{\expcos}{\mathrm{c}}

\newtheorem{lemma}{Lemma}
\newtheorem{proposition}{Proposition}

\title{Normal Vectors on Modified Hopf Manifolds of Delay Differential Equations}
\author{Jonas Otten}
\author{Martin M\"onnigmann\thanks{Corresponding author. Email: martin.moennigmann@rub.de}}
\affil{{\small{\em Automatic Control and Systems Theory, Ruhr-Universit\"at Bochum, Bochum, Germany}}}
\date{September 15, 2016}

\begin{document}

\maketitle
\begin{abstract}
This document states the normal vector system for modified Hopf boundaries of delay differential systems with state and parameter dependent delays. Specifically, it states the proof for Proposition~1 in the paper entitled \textit{Robust optimization of delay differential equations with state
      and parameter dependent delays} by the same authors~\cite{Otten2016b}. 
\end{abstract}

\section{Introduction}
\subsection{System Class}

We consider continuous time systems with uncertain and state dependent delays
\begin{equation}\label{eq:sys}
  \dot \x(t)=\f(\x(t),\x(t-\tau_1),\dots,\x(t-\tau_m),\balpha)
\end{equation}
with state vector $\x \in \R^n$, uncertain parameters $\balpha \in \R^{n_\alpha}$, $m$ delays $\tau_i$ and a smooth $\f$ that maps from $\R^{n(m+1)}\times\R^{n_\alpha}$, or an open subset thereof, into $\R^n$. 
The delays may be functions of the current state and uncertain parameters,
\begin{equation}\label{eq:delay}
\tau_i=\tau_i(\x(t),\alpha)\,.
\end{equation}

\subsection{Preliminaries}

The solution to a set of nonlinear equations will be called regular, if the Jacobian of the equations evaluated at this solution has full rank. 

We introduce the following abbreviations. 
\begin{subequations}
\begin{align}
\expsin(\sigma,\omega,\tau)&=\exp(-\sigma\tau)\sin(\omega\tau)\nonumber\\
\expcos(\sigma,\omega,\tau)&=\exp(-\sigma\tau)\cos(\omega\tau)\,. \nonumber
\end{align}
\end{subequations}

Furthermore, we introduce the delay $\tau_0=0$ to simplify the notation. This permits to replace $x(t)$ by $x(t-\tau_0)$ and state expressions for $x(t-\tau_i)$, $i= 0, \dots, m$ instead of for $x(t)$ and for $x(t-t_i)$, $i= 1, \dots, m$ separately.  

The Jacobian of the right hand side of \eqref{eq:sys} w.r.t $\x(t-\tau_i)$ are denoted by $A_i$. All Jacobians are assumed to be evaluated at a steady state if not noted otherwise. We denote such a steady state $\xE$. 

We refer to $\lambda \in \mathbb{C}$ as an eigenvalue at a steady state $\xE$ of \eqref{eq:sys}, if
\begin{equation}\label{eq:det}
\mathrm{det}\left(\lambda \I - A_0 -\sum_{i=1}^mA_i\exp(-\lambda \tau_i)\right)=0
\end{equation}
(cf. \cite{Engelborghs2000}).

By \textit{modified Hopf point} we refer to 
a steady state that has a leading complex conjugate pair of eigenvalues with nonzero real part. A manifold of such points is called modified Hopf manifold.

\section{Augmented system of Modified Hopf Manifold}

\begin{lemma}[augmented system modified Hopf~\cite{Otten2016b}]
Assume $\xCrit$ is a steady state for parameter values $\alphCrit$. 
If $\lambda=\sigma\pm\i\omega$ are the leading eigenvalues at the steady state $\xCrit$, then there exist vectors $\ba,\bb \in \mathbb{R}^n$ such that $(\xCrit,\alphCrit)$, obey the equations
\begin{subequations}\label{eq:ExpStabManifold}
\begin{align}
\f(\xCrit,\xCrit,...,\xCrit,\alphCrit)&=0\label{eq:ExpStabManifoldSteady}\\
\sigma\ba-\omega\bb-\sum_{i=0}^m A_i\big(\expcos(\sigma,\omega,\tau_i)\ba+\expsin(\sigma,\omega,\tau_i)\bb\big)&=0\label{eq:ExpStabManifoldReal}\\
\omega\ba+\sigma\bb-\sum_{i=0}^m A_i\big(\expcos(\sigma,\omega,\tau_i)\bb-\expsin(\sigma,\omega,\tau_i)\ba\big)&=0\label{eq:ExpStabManifoldImag}\\
\ba^\prime\ba+\bb^\prime\bb -1&=0\label{eq:ExpStabManifoldLength}\\
\ba^\prime\bb&=0\label{eq:ExpStabManifoldPhase}\,.
\end{align}
\end{subequations}
\end{lemma}

\section{Normal Vector on Modified Hopf Manifold}
The following proposition is stated in~\cite{Otten2016b}, but only a sketch of the proof was given there. It is the purpose of the present document to state a complete proof.
\begin{proposition}[normal vectors modified Hopf]
If  $(\xCrit,\alphCrit,\omega,\ba,\bb)$ is a regular solution to \eqref{eq:ExpStabManifold} for an arbitrary but fixed $\sigma$, then $\r$ that obeys the following equations is normal to the manifold of modified Hopf points at this solution:
\begin{subequations}\label{eq:ExpStabNV}
\begin{align}
\text{equations \eqref{eq:ExpStabManifold}}\label{eq:ExpStabNV1}\\
\begin{bmatrix}
\nabla_{\xCrit}\f^\prime&\B_{12}&\B_{13}&0&0\\
0&\B_{22}&\B_{23}&2\ba&\bb\\
0&\B_{32}&\B_{33}&2\bb&\ba\\
0&\B_{42}&\B_{43}&0&0\\
\end{bmatrix}\bkappa&=0\label{eq:ExpStabNV2}\\
\begin{bmatrix}
\nabla_{\alphCrit}\f^\prime&\B_{52}&\B_{53}&0&0
\end{bmatrix}\bkappa-\r&=0\label{eq:ExpStabNV3}\\
\r^\prime \r-1 &=0\label{eq:NVNorming}
\end{align}
\end{subequations}
where
\begin{align*}
\B_{12}=&\sum_{i=0}^m \sigma (\nabla_{\xCrit}\tau_i)\big(\expcos(\sigma,\omega,\tau_i)\ba^\prime+\expsin(\sigma,\omega,\tau_i)\bb^\prime\big)A_i^\prime\\
&-\sum_{i=0}^m \omega(\nabla_{\xCrit}\tau_i)\big(\expcos(\sigma,\omega,\tau_i)\bb^\prime-\expsin(\sigma,\omega,\tau_i)\ba^\prime\big)A_i^\prime\\
&-\sum_{i=0}^m \expcos(\sigma,\omega,\tau_i)(\nabla_{\xCrit}\ba^\prime A_i^\prime)+\expsin(\sigma,\omega,\tau_i)(\nabla_{\xCrit}\bb^\prime A_i^\prime)\,,\\
\B_{13}=&\sum_{i=0}^m  \sigma(\nabla_{\xCrit}\tau_i)\big(\expcos(\sigma,\omega,\tau_i)\bb^\prime-\expsin(\sigma,\omega,\tau_i)\ba^\prime\big)A_i^\prime\\
&+\sum_{i=0}^m \omega(\nabla_{\xCrit}\tau_i) \big(\expsin(\sigma,\omega,\tau_i)\bb^\prime+\expcos(\sigma,\omega,\tau_i)\ba^\prime)A_i^\prime\\
&-\sum_{i=0}^m \expcos(\sigma,\omega,\tau_i)(\nabla_{\xCrit}\bb^\prime A_i^\prime)-\expsin(\sigma,\omega,\tau_i)(\nabla_{\xCrit}\ba^\prime A_i^\prime)\,,
\end{align*}
\begin{align*}
\B_{22}=&\sigma I-\sum_{i=0}^m \expcos(\sigma,\omega,\tau_i)A_i^\prime\,,\\
\B_{23}=&\omega I+\sum_{i=0}^m \expsin(\sigma,\omega,\tau_i)A_i^\prime\,,
\end{align*}
\begin{align*}
\B_{32}=&-\omega I-\sum_{i=0}^m \expsin(\sigma,\omega,\tau_i)A_i^\prime\,,\\
\B_{33}=&\sigma I-\sum_{i=0}^m \expcos(\sigma,\omega,\tau_i)A_i^\prime\,,
\end{align*}
\begin{align*}
\B_{42}=&-\bb^\prime+\sum_{i=0}^m\tau_i\big(\expsin(\sigma,\omega,\tau_i)\ba^\prime-\expcos(\sigma,\omega,\tau_i)\bb^\prime)\big)A_i^\prime\,,\\
\B_{43}=&\ba^\prime+\sum_{i=0}^m \tau_i\big(\expsin(\sigma,\omega,\tau_i)\bb^\prime+\expcos(\sigma,\omega,\tau_i)\ba^\prime)\big)A_i^\prime\,,
\end{align*}
\begin{align*}
\B_{52}=&\sum_{i=0}^m \sigma(\nabla_{\alphCrit}\tau_i)\big(\expcos(\sigma,\omega,\tau_i)\ba^\prime+\expsin(\sigma,\omega,\tau_i)\bb^\prime\big)A_i^\prime\\
&+\sum_{i=0}^m \omega(\nabla_{\alphCrit}\tau_i)\big(\expsin(\sigma,\omega,\tau_i)\ba^\prime-\expcos(\sigma,\omega,\tau_i)\bb^\prime\big)A_i^\prime\\
&-\sum_{i=0}^m \expcos(\sigma,\omega,\tau_i)(\nabla_{\alphCrit}a^\prime A_i^\prime)+\expsin(\sigma,\omega,\tau_i)(\nabla_{\alphCrit}b^\prime A_i^\prime)\,,\\
\B_{53}=&\sum_{i=0}^m \sigma(\nabla_{\alphCrit}\tau_i)\big(\expcos(\sigma,\omega,\tau_i)\bb^\prime-\expsin(\sigma,\omega,\tau_i)\ba^\prime\big)A_i^\prime\\
&+\sum_{i=0}^m \omega(\nabla_{\alphCrit}\tau_i)\big(\expsin(\sigma,\omega,\tau_i)\bb^\prime+\expcos(\sigma,\omega,\tau_i)\ba^\prime\big)A_i^\prime\\
&-\sum_{i=0}^m \expcos(\sigma,\omega,\tau_i)(\nabla_{\alphCrit}\bb^\prime A_i^\prime)-\expsin(\sigma,\omega,\tau_i)(\nabla_{\alphCrit}\ba^\prime A_i^\prime)\,.
\end{align*}
The expressions $\nabla_{\xCrit}\ba^\prime A_i^\prime$ are given by
\begin{equation*}
(\nabla_{\xCrit}\ba^\prime A_i^\prime)_{\mu,\nu}=\sum_{\rho=1}^n \ba_\rho\frac{\partial^2 f_\nu}{\partial \xE^{(c)}_\mu\,\partial \xCrit_\rho(t-\tau_i)}\,.
\end{equation*}
The expressions $\nabla_{\xCrit}\bb^\prime A_i^\prime$,  $\nabla_{\alphCrit}\ba^\prime A_i^\prime$ and  $\nabla_{\alphCrit}\bb^\prime A_i^\prime$ are defined accordingly.
\end{proposition}

\begin{proof}
Consider \eqref{eq:ExpStabManifold} as $3n+2$ equations in the $3n+n_\alpha+1$ variables $\xCrit$, $\alphCrit$, $\omega$, $\ba$ and $\bb$. In the neighborhood of a regular solution $(\xCrit,\alphCrit,\omega,\ba,\bb) \in \mathbb{R}^n\times\mathbb{R}^{n_\alpha}\times\mathbb{R}\times\mathbb{R}^n\times\mathbb{R}^n$, these equations locally define an $(n_\alpha-1)$-dimensional manifold of modified Hopf points.

The rows of the Jacobian of \eqref{eq:ExpStabManifold} w.r.t. $\xCrit$, $\bb$, $\ba$, $\omega$ and $\alphCrit$ at a regular solution span the normal space to the manifold (cf. \cite{Monnigmann2002}). It is more convenient to work with the transposed Jacobian in the sequel than with the Jacobian itself. The transposed Jacobian is denoted by $\B$ and reads
\begin{align*}
\B=&\begin{bmatrix}
\B_{11}&\B_{12}&\B_{13}&\B_{14}&\B_{15}\\
\B_{21}&\B_{22}&\B_{23}&\B_{24}&\B_{25}\\
\B_{31}&\B_{32}&\B_{33}&\B_{34}&\B_{35}\\
\B_{41}&\B_{42}&\B_{43}&\B_{44}&\B_{45}\\
\B_{51}&\B_{52}&\B_{53}&\B_{54}&\B_{55}\\
\end{bmatrix}\nonumber\\
=&\begin{bmatrix}
\nabla_{\xCrit}\\
\nabla_{\ba}\\
\nabla_{\bb}\\
\nabla_{\omega}\\
\nabla_{\alphCrit}\\
\end{bmatrix}
\begin{bmatrix}
\f(\xCrit,\xCrit,...,\xCrit,\alphCrit)\\
\sigma\ba-\omega\bb-\sum_{i=0}^m A_i\big(\expcos(\sigma,\omega,\tau_i)\ba+\expsin(\sigma,\omega,\tau_i)\bb\big)\\
\omega\ba+\sigma\bb-\sum_{i=0}^m A_i\big(\expcos(\sigma,\omega,\tau_i)\bb-\expsin(\sigma,\omega,\tau_i)\ba\big)\\
\ba^\prime\ba+\bb^\prime\bb -1\\
\ba^\prime\bb
\end{bmatrix}^\prime\,.
\end{align*} 
The first, fourth and fifth column of $B$ contain the Jacobians
of  $\f^\prime(\xCrit,\xCrit,...,\xCrit,\alphCrit)$, $\ba^\prime\ba+\bb^\prime\bb -1$ and  $\ba^\prime\bb$, respectively. They result from simple calculations which are not detailed here. 

The calculations required to determine $B_{22}$, $B_{23}$, $B_{32}$ and $B_{33}$, which correspond to the 
Jacobians of $\sigma\ba^\prime-\omega\bb^\prime-\sum_{i=0}^m \big(\expcos(\sigma,\omega,\tau_i)\ba^\prime-\expsin(\sigma,\omega,\tau_i)\bb^\prime\big)A_i^\prime$ and
$\omega\ba^\prime+\sigma\bb^\prime-\sum_{i=0}^m\big(\expcos(\sigma,\omega,\tau_i)\bb^\prime+\expsin(\sigma,\omega,\tau_i)\ba^\prime\big) A_i^\prime$  with respect to $\ba$, $\bb$, and $\omega$,
are also simple. They result in the following expressions:
\begin{align*}
\B_{22}=&\nabla_{\ba}\left(\sigma\ba^\prime-\omega\bb^\prime-\sum_{i=0}^m \big(\expcos(\sigma,\omega,\tau_i)\ba^\prime-\expsin(\sigma,\omega,\tau_i)\bb^\prime\big)A_i^\prime\right)\\
=&\sigma I-\sum_{i=0}^m \expcos(\sigma,\omega,\tau_i)A_i^\prime\\
\B_{23}=&\nabla_{\ba}\left(\omega\ba^\prime+\sigma\bb^\prime-\sum_{i=0}^m\big(\expcos(\sigma,\omega,\tau_i)\bb^\prime+\expsin(\sigma,\omega,\tau_i)\ba^\prime\big) A_i^\prime\right)\\
=&\omega I+\sum_{i=0}^m \expsin(\sigma,\omega,\tau_i)A_i^\prime\,,
\end{align*}
\begin{align*}
\B_{32}=&\nabla_{\bb}\left(\sigma\ba^\prime-\omega\bb^\prime-\sum_{i=0}^m \big(\expcos(\sigma,\omega,\tau_i)\ba^\prime-\expsin(\sigma,\omega,\tau_i)\bb^\prime\big)A_i^\prime\right)\\
=&-\omega I-\sum_{i=0}^m \expsin(\sigma,\omega,\tau_i)A_i^\prime\\
\B_{33}=&\nabla_{\bb}\left(\omega\ba^\prime+\sigma\bb^\prime-\sum_{i=0}^m\big(\expcos(\sigma,\omega,\tau_i)\bb^\prime+\expsin(\sigma,\omega,\tau_i)\ba^\prime\big) A_i^\prime\right)\\
=&\sigma I-\sum_{i=0}^m \expcos(\sigma,\omega,\tau_i)A_i^\prime\,,
\end{align*}
\begin{align*}
\B_{42}=&\nabla_{\omega}\left(\sigma\ba^\prime-\omega\bb^\prime-\sum_{i=0}^m \big(\expcos(\sigma,\omega,\tau_i)\ba^\prime-\expsin(\sigma,\omega,\tau_i)\bb^\prime\big)A_i^\prime\right)\\
=&-\bb^\prime+\sum_{i=0}^m\tau_i\big(\expsin(\sigma,\omega,\tau_i)\ba^\prime-\expcos(\sigma,\omega,\tau_i)\bb^\prime)\big)A_i^\prime\\
\B_{43}=&\nabla_{\omega}\left(\omega\ba^\prime+\sigma\bb^\prime-\sum_{i=0}^m\big(\expcos(\sigma,\omega,\tau_i)\bb^\prime+\expsin(\sigma,\omega,\tau_i)\ba^\prime\big) A_i^\prime\right)\\
=&\ba^\prime+\sum_{i=0}^m \tau_i\big(\expsin(\sigma,\omega,\tau_i)\bb^\prime+\expcos(\sigma,\omega,\tau_i)\ba^\prime)\big)A_i^\prime\,,
\end{align*}

The Jacobians w.r.t $\xCrit$ and $\alphCrit$ call for more attention. Recall that $\tau_i$ is, as described in \eqref{eq:delay}, a function of $\xCrit$ and $\alphCrit$. In this case, the product rule has to be applied to 
three expression depending on $\xCrit$, the exponential function, a trigonometric function and $A_i^\prime$.

\begin{align*}
\B_{12}=&\nabla_{\xCrit} \left(\sigma\ba^\prime-\omega\bb^\prime-\sum_{i=0}^m \big(\expcos(\sigma,\omega,\tau_i)\ba^\prime-\expsin(\sigma,\omega,\tau_i)\bb^\prime\big)A_i^\prime \right)\\
=&\sum_{i=0}^m \sigma(\nabla_{\xCrit}\tau_i)\big(\expcos(\sigma,\omega,\tau_i)\ba^\prime+\expsin(\sigma,\omega,\tau_i)\bb^\prime\big)A_i^\prime\\
&-\sum_{i=0}^m \omega(\nabla_{\xCrit}\tau_i)\big(\expcos(\sigma,\omega,\tau_i)\bb^\prime-\expsin(\sigma,\omega,\tau_i)\ba^\prime\big)A_i^\prime\\
&-\sum_{i=0}^m \expcos(\sigma,\omega,\tau_i)(\nabla_{\xCrit}\ba^\prime A_i^\prime)+\expsin(\sigma,\omega,\tau_i)(\nabla_{\xCrit}\bb^\prime A_i^\prime))\\
\B_{13}=&\nabla_{\xCrit} \left(\omega\ba^\prime+\sigma\bb^\prime-\sum_{i=0}^m\big(\expcos(\sigma,\omega,\tau_i)\bb^\prime+\expsin(\sigma,\omega,\tau_i)\ba^\prime\big) A_i^\prime\right)\\
=&\sum_{i=0}^m  \sigma(\nabla_{\xCrit}\tau_i)\big(\expcos(\sigma,\omega,\tau_i)\bb^\prime-\expsin(\sigma,\omega,\tau_i)\ba^\prime\big)A_i^\prime\\
&+\sum_{i=0}^m \omega(\nabla_{\xCrit}\tau_i) \big(\expsin(\sigma,\omega,\tau_i)\bb^\prime+\expcos(\sigma,\omega,\tau_i)\ba^\prime)A_i^\prime\\
&-\sum_{i=0}^m \expcos(\sigma,\omega,\tau_i)(\nabla_{\xCrit}\bb^\prime A_i^\prime)-\expsin(\sigma,\omega,\tau_i)(\nabla_{\xCrit}\ba^\prime A_i^\prime))\,.
\end{align*}

The product rule also has to be applied when calculating the Jacobians w.r.t. $\alphCrit$:

\begin{align*}
\B_{52}=&\nabla_{\alphCrit} \left(\sigma\ba^\prime-\omega\bb^\prime-\sum_{i=0}^m \big(\expcos(\sigma,\omega,\tau_i)\ba^\prime-\expsin(\sigma,\omega,\tau_i)\bb^\prime\big)A_i^\prime \right)\\
=&\sum_{i=0}^m \sigma(\nabla_{\alphCrit}\tau_i)\big(\expcos(\sigma,\omega,\tau_i)\ba^\prime+\expsin(\sigma,\omega,\tau_i)\bb^\prime\big)A_i^\prime\\
&+\sum_{i=0}^m \omega(\nabla_{\alphCrit}\tau_i)\big(\expsin(\sigma,\omega,\tau_i)\ba^\prime-\expcos(\sigma,\omega,\tau_i)\bb^\prime\big)A_i^\prime\\
&-\sum_{i=0}^m \expcos(\sigma,\omega,\tau_i)(\nabla_{\alphCrit}a^\prime A_i^\prime)+\expsin(\sigma,\omega,\tau_i)(\nabla_{\alphCrit}b^\prime A_i^\prime)\\
\B_{53}=&\nabla_{\alphCrit} \left(\omega\ba^\prime+\sigma\bb^\prime-\sum_{i=0}^m\big(\expcos(\sigma,\omega,\tau_i)\bb^\prime+\expsin(\sigma,\omega,\tau_i)\ba^\prime\big) A_i^\prime\right)\\
=&\sum_{i=0}^m \sigma(\nabla_{\alphCrit}\tau_i)\big(\expcos(\sigma,\omega,\tau_i)\bb^\prime-\expsin(\sigma,\omega,\tau_i)\ba^\prime\big)A_i^\prime\\
&+\sum_{i=0}^m \omega(\nabla_{\alphCrit}\tau_i)\big(\expsin(\sigma,\omega,\tau_i)\bb^\prime+\expcos(\sigma,\omega,\tau_i)\ba^\prime\big)A_i^\prime\\
&-\sum_{i=0}^m \expcos(\sigma,\omega,\tau_i)(\nabla_{\alphCrit}\bb^\prime A_i^\prime)-\expsin(\sigma,\omega,\tau_i)(\nabla_{\alphCrit}\ba^\prime A_i^\prime)\\
\end{align*}

It remains to state expressions for the
vector matrix products such as $\nabla_{\xCrit}\ba^\prime A_i^\prime$. In order to do so we have to switch to components. 
The matrix $A_i$ contains 
\begin{equation*}
(A_i)_{\rho,\nu}=\frac{\partial f_\rho}{\partial \xCrit_\nu(t-\tau_i)}\,
\end{equation*}
in its $\rho$-th row and $\nu$-th column. 
Multiplying it with a vector $\ba$ from the right results in $A_i a$ with components
\begin{equation*}
(A_i\ba)_\nu = \sum_{\rho=1}^n \ba_\rho \frac{\partial f_\nu}{\partial \xCrit_\rho(t-\tau_i)}\,.
\end{equation*}
Transposing and calculating the required derivative yields the matrix $\nabla_{\xCrit}\ba^\prime A_i^\prime$ with components
\begin{equation*}
(\nabla_{\xCrit}\ba^\prime A_i^\prime)_{\mu,\nu}=\sum_{\rho=1}^n \ba_\rho\frac{\partial^2 f_\nu}{\partial \xE^{(c)}_\mu\,\partial \xCrit_\rho(t-\tau_i)}\,
\end{equation*}
in its $\mu$-th row and $\nu$-th column. 
The other derivatives of vector matrix products can be found accordingly.

The matrix whose columns span the normal vector space is now completely determined. We are looking for the parameter space normal vector space, thus we have to find the kernel $\bkappa$ of the first $3n+1$ rows of the transposed Jacobian,
\begin{equation*}
\begin{bmatrix}
\nabla_{\xCrit}\f^\prime&\B_{12}&\B_{13}&0&0\\
0&\B_{22}&\B_{23}&2\ba&\bb\\
0&\B_{32}&\B_{33}&2\bb&\ba\\
0&\B_{42}&\B_{43}&0&0\\
\end{bmatrix}\bkappa=0
\end{equation*}
 which leads to \eqref{eq:ExpStabNV2}. By multiplying $\bkappa$ with the last $n_\alpha$ rows of $\B$ we get the parameter space component of the normal space,
\begin{equation*}
\r=\begin{bmatrix}
\nabla_{\alphCrit}\f^\prime&\B_{52}&\B_{53}&0&0
\end{bmatrix}\bkappa\,,
\end{equation*}
which yields \eqref{eq:ExpStabNV3}.
The length of $\r$ in not determined yet, \eqref{eq:NVNorming} fixes the length of $\r$ to unit length.
\end{proof}

\bibliographystyle{IEEEtran}
\bibliography{DDESupplyChainNV}   
\end{document}